\documentclass[11pt]{amsart}
\usepackage{tabularx,booktabs}
\usepackage{caption}
\usepackage{amsmath}
\usepackage{amsfonts}
\usepackage{bbold}

\usepackage{amscd}
\usepackage{amsthm}
\usepackage{amssymb} \usepackage{latexsym}
\usepackage{eufrak}
\usepackage{euscript}
\usepackage{epsfig}
\usepackage{graphics}
\usepackage{array}
\usepackage{enumerate}
\usepackage{dsfont}
\usepackage{color}
\usepackage{wasysym}
\usepackage{hyperref}
\usepackage{pdfsync}

\usepackage{graphicx}
\usepackage{float}
\usepackage{subfigure}

\usepackage{tikz}

\newcommand{\bel}[1]{\begin{equation}\label{#1}}

\newcommand{\be}{\begin{equation}}

\newcommand{\ba}{\begin{eqnarray}}
\newcommand{\ea}{\end{eqnarray}}

\newcommand{\qe}{\end{equation}}
\newcommand{\R}{{\mathbb R}}
\newcommand{\N}{{\mathbb N}}

\newcommand{\diam}{\mathrm{diam}}

\newcommand{\Hmm}[1]{\leavevmode{\marginpar{\tiny%
$\hbox to 0mm{\hspace*{-0.5mm}$\leftarrow$\hss}%
\vcenter{\vrule depth 0.1mm height 0.1mm width \the\marginparwidth}%
\hbox to
0mm{\hss$\rightarrow$\hspace*{-0.5mm}}$\\\relax\raggedright #1}}}

\newcommand{\width}{\mathrm{width}}

\newtheorem{theorem}{Theorem}[section]

\newtheorem{lemma}[theorem]{Lemma}
\newtheorem{corollary}[theorem]{Corollary}
\newtheorem{definition}[theorem]{Definition}

\newtheorem{remark}[theorem]{Remark}

\newtheorem{prop}[theorem]{Proposition}

\newtheorem{example}[theorem]{Example}

\begin{document}

\title[Discrete infinity Laplace and tug-of-war games]{Discrete infinity Laplace equations on graphs and tug-of-war games}

\author{Fengwen Han}
\address{Fengwen Han: School of Mathematics and Statistics, Henan University, 475004 Kaifeng, Henan, China}
\email{\href{mailto:fwhan@outlook.com}{fwhan@henu.edu.cn}}


\author{Tao Wang}
\address{Tao, Wang: School of Mathematical Sciences, Fudan University, Shanghai 200433, China}
\email{\href{mailto:taowang21@m.fudan.edu.cn}{taowang21@m.fudan.edu.cn}}

\begin{abstract}
    We study the Dirichlet problem of the following discrete infinity Laplace equation on a subgraph with finite width
    \begin{equation*}
        \Delta_{\infty}u(x):=\inf_{y\sim x}u(y)+\sup_{y\sim x}u(y)-2u(x)=f(x).
    \end{equation*}
    We say that a subgraph has finite width if the distances from all vertices to the boundary are uniformly bounded. By Perron's method, we show the existence of bounded solutions. We also prove the uniqueness if $f\ge 0$ or $f\le 0$ by establishing a comparison result, and hence obtain the existence of game values for corresponding tug-of-war games introduced by Peres, Schramm, Sheffield, and Wilson (2009). As an application we show a strong Liouville property for infinity harmonic functions. 
    
    By an argument of Arzel\`a–Ascoli, we prove the convergence of solutions of $\varepsilon$-tug-of-war games as $\varepsilon \to 0$. Correspondingly, we obtain the existence of bounded solutions to normalized infinity Laplace equations on Euclidean domains with finite width.
\end{abstract}
\maketitle

\par
\maketitle

\bigskip

\section{Introduction}
In 2009, Peres, Schramm, Sheffield, and Wilson \cite{peres2009tug} introduced a ``tug-of-war'' game, which is a two-player, zero-sum, stochastic game; see Section 2 for a brief introduction. The following discrete infinity Laplace equation on a graph $G=(V,E)$ is intensively studied by a probalilistic method in their paper. We write $x\sim y$ if there exists an edge adjoining $x,y\in V$.
\begin{equation}
    \label{equ:DiscreteInfinityLaplaceEquation}
    \begin{cases}
        \Delta_{\infty}u(x)=f(x), \quad &x\in X \subset V,\\
        u(x)=g(x), & x \in Y = V\setminus X,
    \end{cases}
\end{equation}
where $f$ and $g$ are bounded functions on $X$ and $Y$ respectively, and
\begin{equation}
    \Delta_{\infty}u(x):=\inf\limits_{y\sim x}u(y)+\sup\limits_{y\sim x}u(y)-2u(x)
\end{equation}
is called the discrete infinity Laplacian.

In the first part of this paper, we study the existence and uniqueness of bounded solutions to \eqref{equ:DiscreteInfinityLaplaceEquation}. For the homogeneous case, i.e., $f\equiv 0$, on a locally finite graph, the existence and uniqueness of bounded solutions is well-known; see \cite{lazarus1995richman,lazarus1999combinatorial} for a constructive proof. By a probabilistic method, for any graph, Peres et. al. \cite{peres2009tug} proved the existence and uniqueness result for $f\equiv 0$, $\inf f>0$, or $\sup f <0$. If $\inf f=0$ and $X$ is unbounded, or $f$ is sign-changing, there were counterexamples for uniqueness. In this paper, we consider the problem under the assumption that $X$ has ``finite width'', i.e. $\width(X) := \sup\{d(x, X^c): x \in X\} < +\infty$, where $d(x, X^c)$ is the combinatorial distance between $x$ and $X^c$. Moreover, we do not assume that $G$ is locally finite. We prove the existence by Perron's method for bounded $f$ and $g$. If additionally $f\ge 0$ or $f\le 0$ holds, the uniqueness result is obtained. Note that our uniqueness result does not contradict the counterexample in \cite{peres2009tug}, since $X$ in the counterexample does not have ``finite width".
\begin{theorem}
    \label{thm:ExistenceAndUniqueness}
    Let $G=(V,E)$ be a graph, $X\subset V$ with $\width(X)<+\infty$, $f\in L^{\infty}(X)$ and $g \in L^{\infty}(V\setminus X)$. Then the discrete infinity Laplace equation \eqref{equ:DiscreteInfinityLaplaceEquation} admits a bounded solution. Moreover, the bounded solution is unique if $f \ge 0$ or $f\le 0$. 
\end{theorem}
We should point out that for a variation of $\varepsilon$-tug-of-war game, Armstrong and Smart \cite{armstrong2012finite} proved the existence result for bounded $X$ and continuous $f$ and $g$. They also proved the uniqueness if $f\ge 0$ or $f\le 0$ in the same paper. Liu and Schikorra \cite{liu2013game} proved the result for $\inf f>0$ or $\sup f<0$ by a game tree approach.

The following comparison result plays a crucial role in proving Theorem \ref{thm:ExistenceAndUniqueness}.
\begin{theorem}
    \label{thm:CompareResult}
    Let $G=(V,E)$ be a graph, $X \subset V$ with $\width(X) < +\infty$, $f \in \R^{X}$ with $f\ge 0$, and $u, v\in \R^V$ with $v\ge -C$, $u\le C$ for some $C>0$. Suppose that
    \begin{equation*}
            -\Delta_{\infty}v(x) \ge f(x) \ge -\Delta_{\infty}u(x), \quad \forall x \in X.
    \end{equation*}
    Then
    \begin{equation*}
        \sup_{V}(u-v) = \sup_{V\setminus X}(u-v).
    \end{equation*}
\end{theorem}
Le Gruyer\cite{le2007absolutely} proved the result on finite graphs. The result was proved in \cite{armstrong2012finite} for continuous $u$ and $v$ on a variation of $\varepsilon$-tug-of-war game. Generally, the result was proved by contradiction via finding some $x_0 \in X$ with $u(x_0)-v(x_0)=\sup\limits_{X}(u-v) > \sup\limits_{V\setminus X}(u-v)$. However, the supremum may not be attained in our setting. By noting that for any $\varepsilon > 0$, there exists $x_{\varepsilon} \in X$ such that $u(x_{\varepsilon}) - v(x_{\varepsilon}) \geq \sup\limits_{X}(u-v) - \varepsilon$, we also get a contradiction in that case. For comparing results in the continuous setting, we refer to \cite{jensen1993uniqueness,lu2008inhomogeneous,lu2008pde,armstrong2010easy}.

One of the most interesting contributions in \cite{peres2009tug} is that the tug-of-war game presents a probalilistic interpretation to the following continuous normalized infinity Laplace equation
\begin{equation}
    \begin{cases}
    \label{equ:ContinuousInfinityLaplaceEquation}
        \Delta_{\infty}^Nu(x)=f(x), \quad &x\in \Omega \subset \R^n,\\
        u(x)=g(x), & x \in \partial \Omega.
    \end{cases}
\end{equation}
For $u \in C^2$ and $\nabla u(x) \ne 0$, the normalized infinity Laplacian in \eqref{equ:ContinuousInfinityLaplaceEquation} is defined via
\begin{equation}
    \Delta_{\infty}^Nu(x):=\frac{1}{|\nabla u(x)|^2}\sum_{i,j}u_{x_i}(x)u_{x_ix_j}(x)u_{x_j}(x),
\end{equation}

When $f\equiv 0$, the equation \eqref{equ:ContinuousInfinityLaplaceEquation} is equivalent to the (non-normalized) infinity Laplace equation in the viscosity sense, 
\begin{equation}
    \begin{cases}
    \label{equ:infinity_Laplace_equation}
        \Delta_{\infty}u(x)=f(x), \quad &x\in \Omega \subset \R^n,\\
        u(x)=g(x), & x \in \partial \Omega,
    \end{cases}
\end{equation}
where the infinity Laplacian (we use the same symbol as the discrete infinity Laplacian) is defined via
\begin{equation}
    \Delta_{\infty}u(x):=\sum_{i,j}u_{x_i}(x)u_{x_ix_j}(x)u_{x_j}(x).
\end{equation}
The unique viscosity solution $u$ of \eqref{equ:ContinuousInfinityLaplaceEquation} or \eqref{equ:infinity_Laplace_equation} with $f\equiv 0$ is exactly the \emph{absolutely minimal Lipschitz extension} of $g$, i.e. $\mathrm{Lip}_{U}u=\mathrm{Lip}_{\overline{U}}u$ for any open set $U\subset \Omega$. We refer to \cite{aronsson1967extension,jensen1993uniqueness,barles2001existence,aronsson2004tour} for more details. Lu and Wang \cite{lu2008pde} proved the existence and uniqueness of \eqref{equ:ContinuousInfinityLaplaceEquation} in the case that $\inf f >0$ or $\sup f < 0$ on a bounded open subset of $\R^n$ by a Perron's method. For (non-normalized) infinity Laplace equations, we refer to \cite{lu2008inhomogeneous, lindqvist2016notes} for more existence and uniqueness results. It is worth noting that the uniqueness result fails for equation \eqref{equ:infinity_Laplace_equation} with  sign-changing $f$; see \cite{lu2008inhomogeneous} for a counterexample. 

The $\varepsilon$-tug-of-war game introduced in \cite{peres2009tug} provides a discrete method to study the normalized infinity Laplace equation \eqref{equ:ContinuousInfinityLaplaceEquation}. In fact, given a bounded domain $\Omega\subset \R^n$ and $\varepsilon>0$, a corresponding graph $G_{\varepsilon}=(V,E)$ is constructed via setting $V=\overline{\Omega}$, and $x\sim y$ if and only if $d_{\overline{\Omega}}(x,y) < \varepsilon$, where $d_{\overline{\Omega}}$ is the induced intrinsic metric of $\overline{\Omega}$. Then the solution of the following discrete infinity Laplace equation converges to a solution of \eqref{equ:ContinuousInfinityLaplaceEquation} as $\varepsilon \to 0$
\begin{equation}
    \begin{cases}
        \Delta_{\infty}^{\varepsilon}u_{\varepsilon}(x)=\varepsilon^2 f(x), \quad & x\in \Omega;\\
        u_{\varepsilon}(x)=g(x), & x\in \partial \Omega,
    \end{cases}
\end{equation}
where $\Delta_{\infty}^{\varepsilon}$ defined via $$\Delta_{\infty}^{\varepsilon}u(x):=\inf\limits_{y\in B_x(\varepsilon)}u(y)+\sup\limits_{y\in B_x(\varepsilon)}u(y)-2u(x)$$ is the discrete infinity Laplacian on $G_{\varepsilon}$, $f \in C(\Omega)\cap L^{\infty}(\Omega)$, and $g\in C(\partial\Omega)$. The convergence was proved by \cite{peres2009tug} for $f\equiv 0$, $\inf f >0$, or $\sup f <0$ by a probalilistic method. Armstrong and Smart \cite{armstrong2012finite} introduced a ``boundary-biased'' $\varepsilon$-tug-of-war game, based on which they proved the convergence for all $f \in C(\Omega)\cap L^{\infty}(\Omega)$. We also refer to \cite{oberman2005convergent, le2007absolutely} for $f\equiv 0$, and \cite{peres2010biased} for other settings. Inspired by their excellent works, we will prove the result for the original $\varepsilon$-tug-of-war game on a domain $\Omega \subset \R^n$ with finite width, i.e. $\width(\Omega) := \sup\limits_{x \in \Omega}d_{\overline{\Omega}}(\partial\Omega, x) < +\infty$.

\begin{theorem}\label{thm:Convergence}
    Let $\Omega \subset \R^n$ such that $\width(\Omega) < +\infty$, $f\in L^{\infty}(\Omega)$ and $g\in C(\partial\Omega)\cap L^{\infty}(\partial\Omega)$, $\{\varepsilon_i\}$ be a sequence of positive numbers converging to $0$ as $i \to \infty$, $u_i$ be a bounded solution of
    \begin{equation}
    \label{equ:epsilon-tug-of-war-game}
       \begin{cases}
         \Delta_{\infty}^{\varepsilon_i}u_i(x) = \varepsilon_i^2f(x), \quad & x\in \Omega; \\
         u_i(x) = g(x), & x\in \partial \Omega.
       \end{cases}
    \end{equation}  
Then there exists a subsequence of $\{u_i\}$ converging to some $u\in C(\overline{\Omega})$ locally uniformly.
\end{theorem}
We should point out that there are many domains satisfying the conditions of Theorem \ref{thm:Convergence}. For example, given two functions $f_1, f_2: \mathbb{R}^{n-1} \to \mathbb{R}$, if $\|f_1-f_2\|_{\infty} \leq C < +\infty$ for some constants $C$, then the domain between $\{(x, f_1(x)): x \in \mathbb{R}^{n-1}\}$ and $\{(x, f_2(x)): x \in \mathbb{R}^{n-1}\}$ satisfies the $\width(\Omega)< +\infty$. Another interesting domain is $\mathbb{R}^n \setminus \mathbb{Z}^n$.

As a corollary of Theorem \ref{thm:Convergence}, we prove the existence of bounded solutions to \eqref{equ:ContinuousInfinityLaplaceEquation} on domains with finite widths; See Section 2.

The rest of this paper is organized as follows. In Section 2, we introduce some basic notions and definitions. Section 3 is devoted to some properties of discrete infinity Laplace equation on graphs. Firstly, we prove some basic properties for the discrete infinity Laplacian. Then we prove Theorem~\ref{thm:CompareResult} and Theorem~\ref{thm:ExistenceAndUniqueness}. At the end of this section, we prove a strong Liouville result as an application.  In Section 4, we discuss the necessity of the conditions needed in Theorem~\ref{thm:ExistenceAndUniqueness}. Some examples are presented. In Section 5, we prove the convergence of solutions to $\varepsilon$-tug-of-war games. We prove Theorem~\ref{thm:Convergence} by a modification of the proof of Arzel\`a–Ascoli Theorem.

\section{Preliminaries}
\subsection{Basic notions and definitions}
Through this paper, we write $G=(V,E)$ for an \emph{undirected} graph, where $V$ is the set of vertices and $E$ is the set of edges. Note that we don't assume that $G$ is locally finite. Let $X,Y$ be subsets of $V$, $x,y\in V$. We write the following: 
\begin{itemize}
    \item $x\sim y$ if there exists an edge connecting $x,y\in V$; 
    \item $\partial X:= \{y\notin X: \text{there exists a $x\in X$ such that $y\sim x$}\}$, which is called the \emph{boundary} of $X$;
    \item $d_G(X,Y):=\inf\{n: X\ni x_0 \sim x_1 \sim \cdots \sim x_n \in Y\}$, i.e. the length of a shortest path connecting $X$ and $Y$, which is called the \emph{combinatorial distance} between $X$ and $Y$;
    \item $d_{G}(Y,x):=d_G(Y, \{x\})$ and $d_G(y,x):= d_G(\{y\},\{x\})$;
    \item $\diam(X):=\sup\limits_{x,y\in X}d_G(x,y)$, which is called the \emph{diameter} of $X$;
    \item $\width(X):=\sup\limits_{x\in X} d_G(\partial X,x)$, which is called the \emph{width} of $X$;
    \item $B_x(r):=\{y\in V: d_G(x,y) < r\}$, which is called the \emph{open $r$-ball centered at $x$};
    \item $\R^{X}:=\{u:X\to \R\}$, i.e. the set of all functions on $X$;
    \item $L^{\infty}(X)$ for the set of all bounded functions on $X$.
\end{itemize}

Let $\Omega\subset \R^n$ be a domain. We write the following:
\begin{itemize}
    \item $\partial \Omega$ is the boundary of $\Omega$, and $\overline{\Omega}:=\Omega \cup \partial\Omega$ is the closure of $\Omega$;
    \item $d_{\overline{\Omega}}$ is the induced intrinsic metric of $\overline{\Omega}$, i.e. $d_{\overline{\Omega}}(x,y)$ is the infimum of the lengths of all paths from $x$ to $y$ in $\overline{\Omega}$;
    \item $\width(\Omega):=\sup\limits_{x\in \Omega} d_{\overline{\Omega}}(\partial \Omega,x)$, which is called the \emph{width} of $\Omega$;
    \item $B_x(r):=\{y\in \overline{\Omega}:d_{\overline{\Omega}}(x,y) < r\}$ is the open $r$-ball centered at $x$;
    \item $\Omega_r:= \{x \in \Omega: d_{\overline{\Omega}}(x, \partial \Omega)>r\}$.
\end{itemize}
We ususlly omit the subscripts $G,\overline{\Omega}$ of $d_{G},d_{\overline{\Omega}}$ defined above.
\begin{definition}[Discrete infinity Laplacian]
    Let $G=(V,E)$ be a graph, $u\in \R^V$. The discrete infinity Laplacian $\Delta_{\infty}$ is defined via
    \begin{equation*}
        \Delta_{\infty}u(x):=\inf_{y\sim x}u(y)+\sup_{y\sim x}u(y)-2u(x).
    \end{equation*}
\end{definition}
We write $u_+(x):=\sup\limits_{y\sim x}u(y)$ and $u_-(x):=\inf\limits_{y\sim x}u(y)$ for convenience. Hence
\begin{equation*}
    \Delta_{\infty}u(x)=u_+(x)+u_-(y)-2u(x).
\end{equation*}
We say $u\in \R^V$ is \emph{infinity super-harmonic} (or \emph{infinity sup-harmonic} resp.) if $-\Delta_{\infty}u \ge 0$ (or $-\Delta_{\infty}u \le 0$ resp.), and $u$ is \emph{infinity harmonic} if $u$ is both infinity super-harmonic and sub-harmonic.

\subsection{Tug-of-war game and $\varepsilon$-tug-of-war game}
We give a brief introduction of \emph{tug-of-war game} here. Let $G=(V,E)$ be a graph, $V=X \sqcup Y $, $f\in \R^X$ which serves as the running payoff function, and $g\in \R^Y$ which serves as the terminal payoff function. A token is placed at some initial place $x_0 \in X$, and two players play this game. At the $k$-th round, a fair coin is tossed to determine the winner of this round, who can move the token to a favorable place $x_{k+1}$ adjacent to the current position $x_k$. When the token arrives at some vertex $y\in Y$ right after the $n$-th round, the game ends and player I get a payoff of
\begin{equation}
    R(x_0)=\sum_{k=0}^{n} f(x_k) + g(y),
\end{equation}
which is also the debt of player II. If the game never ends, then player I gets a payoff of $-\infty$ and player II gets a debt of $+\infty$. Play I tries to maximize its payoff and Play II tries to minimize its debt. Let $\mathcal{S}$ be the set of all strategies. By saying a strategy, we mean a map $x_0 \sim x_1 \sim \cdots \sim x_k \mapsto x_{k+1}\sim x_k$, i.e. the favorable position $x_{k+1}$ at the $k$-th round is determined by the positions of the token at the first $k-1$ rounds. Given the strategies of player I and II $s_{I}$ and $s_{II}$, the move of the token is a stochastic process determined by the result of the fare coin tossed each round. Let $E_{I}(s_{I},s_{II})$ and $E_{II}(s_{I},s_{II})$ be the expected payoffs of player I and player II respectively. Then define 
\begin{equation}
    u_{I}(x_0):=\sup_{s_I\in \mathcal{S}}\inf_{s_{II}\in \mathcal{S}}E_{I}(s_{I},s_{II}), \ u_{II}(x_0):=\inf_{s_{II}\in \mathcal{S}}\sup_{s_{I}\in \mathcal{S}}E_{II}(s_{I},s_{II}), \ x_0\in X,
\end{equation}
as the value of player I and player II at $x_0$ respectively. These functions satisfy the following discrete infinity Laplace equation \cite{peres2009tug}:
\begin{equation}
    \Delta_{\infty}u_{I}(x)=\Delta_{\infty}u_{II}(x)=-2f(x), \ x\in X.
\end{equation}
If $u_{I}(x)=u_{II}(x)$ for all vertices $x \in X$, then $u:=u_{I}=u_{II}$ is called the \emph{value} function of the game, and we say the game has a \emph{game value}. 
%

Let $s$ be such a strategy: at the $k$-th round, let $x_{k+1}\in V$ such that $d_{G}(Y,x_{k+1})<d_{G}(Y,x_{k})$, i.e. always move the token towards the boundary $Y$. Since $d_G(Y,x)$ are uniformly bounded, by setting $s_{I}=s$, we have $\inf u_{I} > -\infty$; and by setting $s_{II}=s$, we have $\sup u_{I} < +\infty$. Hence $u_{I} \in L^{\infty}(V)$. Similarly, $u_{II} \in L^{\infty}(V)$. Then by Theorem~\ref{thm:ExistenceAndUniqueness} we know the corresponding tug-of-war game has a game value.
\begin{corollary}
    Let $G=(V,E)$ be a graph, $X\subset V$ with $\width(X)<+\infty$, $f\in L^{\infty}(X)$ with $f \leq 0$ or $f \geq 0$, and $g \in L^{\infty}(V\setminus X)$. Then the corresponding tug-of-war game has a bounded game value. 
\end{corollary}

\subsection{Normalized infinity Laplace equations}
For a function $\phi \in C^2(\Omega)$, define
\begin{align*}
    \Delta_{\infty}^{+}\phi(x):=
    \begin{cases}
        |\nabla\phi(x)|^{-2}\langle\nabla^2\phi(x)\cdot\nabla\phi(x),\nabla\phi(x)\rangle, \ &\text{ if }\nabla\phi(x)\ne 0,\\
        \max\{\langle\nabla^2\phi(x)v,v\rangle:\ |v|=1\}, &\text{ otherwise}.
    \end{cases}
\end{align*}
and
\begin{align*}
    \Delta_{\infty}^{-}\phi(x):=
    \begin{cases}
        |\nabla\phi(x)|^{-2}\langle\nabla^2\phi(x)\cdot\nabla\phi(x),\nabla\phi(x)\rangle, \ &\text{ if }\nabla\phi(x)\ne 0,\\
        \min\{\langle\nabla^2\phi(x)v,v\rangle:\ |v|=1\}, &\text{ otherwise}.
    \end{cases}
\end{align*}
Suppose $u\in C(\Omega)$, we say $\Delta_{\infty}^{+}u(x_0) \ge f(x_0)$ in the viscosity sense, if for any polynomial $\phi$ of degree 2 such that $u-\phi$ attains a strict local minimum at $x_0$, we have 
\begin{equation*}
    \Delta_{\infty}^{+}\phi(x_0) \ge f(x_0).
\end{equation*}
Respectively, we say $\Delta_{\infty}^{-}u(x_0) \le f(x_0)$ in the viscosity sense, if for any polynomial $\phi$ of degree 2 such that $u-\phi$ attains a strict local maximum at $x_0$, we have 
\begin{equation*}
    \Delta_{\infty}^{-}\phi(x_0) \le f(x_0).
\end{equation*}
We say $\Delta_{\infty}^Nu(x_0)=f(x_0)$ (in the viscosity sense) if the above two inequalities hold. The following theorem \cite[Theorem 2.11]{armstrong2012finite} shows that the limit of solutions of \eqref{equ:epsilon-tug-of-war-game} is a solution of \eqref{equ:ContinuousInfinityLaplaceEquation}. Though the proof originally was stated for the case of bounded $\Omega$ and ``boundary-biased'' tug-of-war game, it also applies in our setting. We also refer to \cite{peres2009tug} for a probabilistic proof.
\begin{theorem}[\cite{armstrong2012finite}, Theorem 2.11]\label{thm:Armstrong}
    Assume that $f\in C(\Omega)$ and $\{\varepsilon_i\}$ is a sequence of positive numbers converging to $0$ as $i\to \infty$. Suppose 
    \begin{equation}
        -\Delta_{\infty}^{\varepsilon_i}u_i = \varepsilon_i^2 f \text{ in } \Omega_{\varepsilon_i}.
    \end{equation}
    Suppose also that there exists a function $u\in C(\Omega)$ such that $u_i \to u$ locally uniformly in $\Omega$ as $i\to \infty$. Then $u$ is a viscosity subsolution of the inequality
    \begin{equation}
        -\Delta_{\infty}u = f \text{ in } \Omega.
    \end{equation}
\end{theorem}
We obtain the existence of a bounded solution to the normalized infinity Laplace equation \eqref{equ:ContinuousInfinityLaplaceEquation}.
\begin{corollary}
    Let $\Omega \subset \R^n$ be a open domain with $\width(\Omega)<+\infty$. Suppose $f\in C(\Omega)\cap L^{\infty}(\Omega)$, $g\in C(\partial \Omega) \cap L^{\infty} (\partial\Omega)$. Then there exists a $u\in C(\overline{\Omega})\cap L^{\infty}(\overline{\Omega})$ satisfying \eqref{equ:ContinuousInfinityLaplaceEquation}. 
\end{corollary}
\begin{proof}
    The result is obtained directly by applying Theorem~\ref{thm:Convergence} and Theorem~\ref{thm:Armstrong}. 
\end{proof}

\section{Discrete infinity Laplacian equation on graphs}
Note that if $u$ is a solution of \eqref{equ:DiscreteInfinityLaplaceEquation}, then for all $x\sim y$,
\begin{equation}\label{equ:MarchingEquation}
    u(x)-u_{-}(x)=u_{+}(x)-u(x)-f(x) \ge u(y)-u(x)-f(x).
\end{equation}
This simple inequality is applied frequently through this paper. 

\begin{lemma}[Gradient estimate]\label{lem:GradientEstimation}
    Let $u$ be a solution of \eqref{equ:DiscreteInfinityLaplaceEquation}, $x\in X$, and $x\sim y$. Then for any $N\in \N$ with $N \le d(Y,x)$, the following inequality holds.
    \begin{equation}
        u(y)-u(x) \le \frac{u(x)-\inf{u}}{N}+\frac{(N+1)\sup{f}}{2}.
    \end{equation}
\end{lemma}
\begin{proof}
    Let $D=u(y)-u(x)$. For any $\varepsilon>0$, we can choose $\{x_n\}_{n=0}^{N}$ such that $x=x_0\sim x_1 \sim \cdots \sim x_N$, $u(x_{n+1}) \le u_{-}(x_n)+\varepsilon$, and $x_n\in X$ for all $n<N$ . Then by \eqref{equ:MarchingEquation} we have 
    \begin{align*}
        u(x_n)-u(x_{n+1}) 
        \ge& u(x_n)-u_{-}(x_n)-\varepsilon\\
        \ge& u(x_{n-1})-u(x_n)-f(x_n)-\varepsilon\\
        \ge& u(x_{n-1})-u(x_n)-(\sup f+\varepsilon)\\
        \ge& u(x_{n-2})-u(x_{n-1})-2(\sup f+\varepsilon)\\
        \ge& u(x_0)-u(x_1)-n(\sup f+\varepsilon)\\
        \ge& u(y)-u(x)-(n+1)(\sup f+\varepsilon)\\
        =& D-(n+1)(\sup f+\varepsilon).
    \end{align*}
    Thus, $u(x)-u(x_N) \ge ND - \frac{N(N+1)}{2}(\sup f+\varepsilon)$, which implies that
    \begin{equation}
        D \le \frac{u(x)-u(x_N)}{N}+\frac{(N+1)(\sup f+ \varepsilon)}{2}.
    \end{equation}
    Then the result follows by letting $\varepsilon \to 0$.
\end{proof}
Let $a \in \R$, $b,c \ge 0$, we define the quadratic functions
\begin{equation}\label{equ:quadratic_function}
    \underline{q}(r):=a-br+c\frac{r(r-1)}{2},\quad \overline{q}(r):=a+br-c\frac{r(r-1)}{2}.
\end{equation}
Given $V' \subset V$, we define
\begin{equation}
    \underline{Q}(x)=\underline{q}(d(V',x)),\quad \overline{Q}(x)=\overline{q}(d(V',x)), \ x\in V.
\end{equation}
By a simple calculation, we have the following properties.
\begin{prop}
    If $b>Rc$ for some $R\in \N$, then for all $x \notin V'$ with $1 \le d(V',x) \le R$, the following inequalities hold.
    \begin{enumerate}[(i)]
        \item $\underline{Q}(x_2) < \underline{Q}(x_1)$ and $\overline{Q}(x_2) > \overline{Q}(x_1)$ if $d(V',x_2) > d(V',x_1)$.
        \item $\Delta_{\infty} \underline{Q}(x) \ge c$ and $-\Delta_{\infty} \overline{Q}(x) \ge c$.
    \end{enumerate}
\end{prop}

Given $G=(X\sqcup Y,E)$ with $\width(X)=W<+\infty$, $y_0 \in Y$, $R\in \N$, and $g\in \R^Y$, we construct a new graph $G'$ by the following way: For every $y\in Y$, add $R-1$ new vertices $y_1,y_2,\cdots,y_{R-1}$, and $R$ new edges $y_0 \sim y_1 \sim y_2 \sim \cdots \sim y_R = y$. Then we have
\begin{equation}
    d_{G'}(y_0,y)=\min\{d_{G}(y_0,y),R\}, \quad \forall y\in Y,
\end{equation}
and
\begin{equation}
    d_{G'}(y_0,x)\le W+R,\quad \forall x\in X.
\end{equation}
Let $c \ge 0$, $a=\sup\{g(y):y\in Y, \ d_{G}(y_0,y) < R\}$(or $a=\inf\{g(y):y\in Y, \ d_{G}(y_0,y) < R\}$ resp.), $b=\frac{2||g||_{\infty}}{R}+c(W+R)$, we define the functions on $G'$:
\begin{equation}\label{equ:quadratic_function_on_new_graph}
    \overline{w}(x)=\overline{q}(d_{G'}(y_0,x)),\quad \underline{w}(x)=\underline{q}(d_{G'}(y_0,x)),
\end{equation}
where $\overline{q}$ and $\underline{q}$ are defined as in \eqref{equ:quadratic_function}. Then we have
    \begin{enumerate}[(i)]
        \item $\overline{w}(y_0) = \underline{w}(y_0) = g(y_0)$ if $R=1$;
        \item $\Delta_{\infty}\underline{w}(x) \ge c$, $-\Delta_{\infty}\overline{w}(x) \ge c, \ \forall x\in X$;
        \item $\overline{w}(y) \ge g(y) \ge \underline{w}(y), \ \forall y\in Y.$
    \end{enumerate}

\subsection{Proof of Theorem~\ref{thm:CompareResult} and Theorem~\ref{thm:ExistenceAndUniqueness}}
\begin{proof}[Proof of Theorem~\ref{thm:CompareResult}]
    We prove the result by contradiction. Without loss of generality, we assume that $\sup\limits_{x\in V\setminus X}\{u(x)-v(x)\}=0$. Suppose that $\sup\limits_{x\in X}\{u(x)-v(x)\}=a>0$, let $\Tilde{v}=v+a$, then we have
    \begin{enumerate}[(i)]
        \item $\Delta_{\infty}\Tilde{v}(x) = \Delta_{\infty}{v}(x)$;
        \item $\Tilde{v}(x) \ge u(x)$, $\Tilde{v}_{+}(x) \ge u_{+}(x)$, $\Tilde{v}_{-}(x) \ge u_{-}(x)$, $\forall \ x\in X$;
        \item $\Tilde{v}(x) - u(x) \ge a$, $\forall \ x\in \partial X$;
        \item for any $\varepsilon >0$, there exists a $x_{\varepsilon}\in X$ such that $\Tilde{v}(x_{\varepsilon})\le u(x_{\varepsilon})+\varepsilon$.
    \end{enumerate}
    Note that $\Delta_{\infty}\Tilde{v}(x) \le \Delta_{\infty}{u}(x)$ implies that
    \begin{equation}\label{equ:tmp17}
        \Tilde{v}_{+}(x)-u_{-}(x) \le 2(\Tilde{v}(x)-u(x))+u_+(x)-\Tilde{v}_{-}(x)
    \end{equation}
    and
    \begin{equation}\label{equ:tmp18}
        \Tilde{v}_{-}(x)-u_{-}(x) \le 2(\Tilde{v}(x)-u(x))+u_+(x)-\Tilde{v}_{+}(x) \le 2(\Tilde{v}(x)-u(x)).      
    \end{equation}
    Besides, $-\Delta_{\infty}\Tilde{v}(x) \ge f(x) \ge 0$ implies that
    \begin{equation}\label{equ:tmp19}
        \Tilde{v}(x)-\Tilde{v}_{-}(x) \ge 0,
    \end{equation}
    since
    \begin{equation}\label{equ:tmp20}
        \Tilde{v}(x)-\Tilde{v}_{-}(x) \ge \Tilde{v}_{+}(x)-\Tilde{v}(x) \ge \Tilde{v}_{-}(x) -\Tilde{v}(x).
    \end{equation}
    \underline{Step 1}: we prove that for any $x_0 \in X$, if $\Tilde{v}(x_0)-u(x_0) \le \varepsilon$ and $\Tilde{v}(x_0)-\Tilde{v}_{-}(x_0) \le \delta$, then 
    \begin{equation}
        \sup_{x\sim x_0}\{\Tilde{v}(x)-u(x)\} \le 2(\varepsilon+\delta).
    \end{equation}
    In fact, we have
    \begin{align*}
        \sup_{x\sim x_0}\{\Tilde{v}(x)-u(x)\}
        \le& \Tilde{v}_{+}(x_0)-u_{-}(x_0)\\
        \le& 2(\Tilde{v}(x_0)-u(x_0))+u_{+}(x_0)-\Tilde{v}_{-}(x_0)\quad \\
        \le& 2\varepsilon+\Tilde{v}_{+}(x_0)-\Tilde{v}_{-}(x_0)\\
        =& 2\varepsilon+\Tilde{v}_{+}(x_0)-\Tilde{v}(x_0)+\Tilde{v}(x_0)-\Tilde{v}_{-}(x_0)\\
        \le& 2\varepsilon+2(\Tilde{v}(x_0)-\Tilde{v}_{-}(x_0)) \\
        \le& 2\varepsilon+2\delta.
    \end{align*}
    \underline{Step 2}: we prove that for any $x_0\in X$, if $\Tilde{v}(x_0)-u(x_0) \le \varepsilon$ and $\varepsilon < a$, then 
    \begin{equation}
        \Tilde{v}(x_0)-\Tilde{v}_{-}(x_0) \le \frac{4C}{\log_4 \frac{a}{\varepsilon}}.
    \end{equation}
    Let $\delta = \Tilde{v}(x_0)-\Tilde{v}_{-}(x_0)$. If $\delta =0$, then there is nothing to prove. Suppose that $\delta >0$, let $s = \min\{\varepsilon, \delta\}$. 
    Then there exists a $x_1 \sim x_0$ such that 
    $$\Tilde{v}(x_0)-\Tilde{v}(x_1) \ge \Tilde{v}(x_0)-\Tilde{v}_{-}(x_0) - \frac{s}{4} \ge \delta - \frac{s}{4}.$$
    If $x_1 \in \partial X$, we terminate the process. Otherwise there exists a $x_2\sim x_1$ such that
    \begin{align*}
        \Tilde{v}(x_1)-\Tilde{v}(x_2) &\ge \Tilde{v}(x_1)-\Tilde{v}_{-}(x_1)-\frac{s}{8}\\ 
        &\ge \Tilde{v}_{+}(x_1)-\Tilde{v}(x_1)-\frac{s}{8} \\
        &\ge \Tilde{v}(x_0)-\Tilde{v}(x_1) -\frac{s}{8}\\
        &\ge \delta - \frac{s}{4} - \frac{s}{8} \ge \frac{\delta}{2}.
    \end{align*}
    As so on we find a sequence $x_0\sim x_1 \sim \cdots$ such that 
    $$\Tilde{v}(x_i)-\Tilde{v}(x_{i+1}) \ge \Tilde{v}(x_i)-\Tilde{v}_{-}(x_{i})-\frac{s}{2^{i+2}} \ge \frac{\delta}{2}.$$ 
    We assert that the process terminates in finite steps to get a $x_n \in \partial X.$ In fact, we have $2C \ge \Tilde{v}(x_0)-\Tilde{v}(x_n) \ge \frac{n\delta}{2}$, which implies that
    \begin{equation}
        n \le \frac{4C}{\delta}.
    \end{equation}
    On the other hand, since
    \begin{equation}
        \Tilde{v}(x_i)-\Tilde{v}(x_{i+1}) \ge \Tilde{v}(x_i)-\Tilde{v}_{-}(x_{i})-\frac{s}{2^{i+2}},
    \end{equation}
    we have
    \begin{align*}
        a&\le \Tilde{v}(x_{n})-u(x_{n})\\ 
        &\le \Tilde{v}_{-}(x_{n-1})+\frac{s}{2^{n+1}}-u(x_{n})\\
        &\le \Tilde{v}_{-}(x_{n-1})-u_{-}(x_{n-1})+\varepsilon\\
        &\le 2(\Tilde{v}(x_{n-1})-u(x_{n-1}))+\varepsilon\\
        &\le 2(2(\Tilde{v}(x_{n-2})-u(x_{n-2}))+\varepsilon)+\varepsilon\\
        &\le 2^n(\Tilde{v}(x_{0})-u(x_{0}))+2^{n-1}\varepsilon + 2^{n-2}\varepsilon + \cdots + \varepsilon\\
        &\le 2^n\varepsilon+2^{n-1}\varepsilon + 2^{n-2}\varepsilon + \cdots + \varepsilon\\
        &\le 4^{n}\varepsilon,
    \end{align*}
    which implies that $n\ge \log_4{\frac{a}{\varepsilon}}$. Then we have $\Tilde{v}(x_0)-\Tilde{v}_{-}(x_0) \le \frac{4C}{\log_4 \frac{a}{\varepsilon}}.$\\
    \underline{Step 3}: we prove the theorem. For any $\varepsilon$ small enough, we choose $x_{\varepsilon}\in X$ such that $\Tilde{v}(x_{\varepsilon})-u(x_{\varepsilon}) \le \varepsilon$, then $\sup_{x\sim x_{\varepsilon}}\{\Tilde{v}(x)-u(x)\} \to 0$ as $\varepsilon \to 0$, which implies that
    $$\sup_{x\in B_{x_{\varepsilon}}(R)}\{\Tilde{v}(x)-u(x)\} \to 0, \quad \varepsilon \to 0,$$
    for any finite $R\in \N$. On the other hand, choose $R > \width(X)$, then $B_{x_{\varepsilon}}(R) \cap \partial X \ne \varnothing$, which implies that
    $$\sup_{x\in B_{x_{\varepsilon}}(R)}\{\Tilde{v}(x)-u(x)\} \ge a, \quad \forall \varepsilon >0.$$
    Then we get a contradiction and complete the proof.
\end{proof}
\begin{remark}
    With a similar method, one can prove the result for $f\le 0$. However it is not true if $f$ changes sign; see Example~\ref{exa:Signchangingf}.
\end{remark}
\begin{proof}[Proof of Theorem~\ref{thm:ExistenceAndUniqueness}]
    The uniqueness is obtained directly by Theorem~\ref{thm:CompareResult}. We prove the existence of solutions by Perron's method. Let 
    \begin{equation*}
        \mathcal{A}_{f,g} = \{v\in L^{\infty}(V): \Delta_{\infty} v \le f \text{ in } X, \ v\ge g \text{ on } Y\}.
    \end{equation*}
    Take
    \begin{equation*}
        u(x)= \inf_{v\in \mathcal{A}_{f,g}}v(x), \quad x \in V.
    \end{equation*} 
    \underline{Step 1}: we prove that $u$ is well-defined, bounded, and $u|_{Y} = g$.\\
    Fix $y_0\in Y$, let $c = ||f||_{\infty}$, $a=\sup\{g(y):y\in Y, \ d_{G}(y_0,y) < R\}$(or $a=\inf\{g(y):y\in Y, \ d_{G}(y_0,y) < R\}$ resp.), $b=\frac{2||g||_{\infty}}{R}+c(W+R)$, $\overline{w}$ and $\underline{w}$ be defined as in \eqref{equ:quadratic_function_on_new_graph} with $R=1$, then
    \begin{enumerate}[(i)]
        \item $\overline{w} \in \mathcal{A}_{f,g}$, which implies that $\mathcal{A}_{f,g} \ne \varnothing$, i.e. $u$ is well-defined;
        \item $\overline{w}(y_0)=g(y_0)$, which implies that $u|_{Y} = g$;
        \item $\overline{w}$ is bounded from above, which implies that $u$ is bounded from above by the definition of $u$;
        \item $\underline{w}$ is bounded from below, which implies that $u$ is bounded from below by Theorem~\ref{thm:CompareResult}.
    \end{enumerate}
    \underline{Step 2}: we prove that $\Delta_{\infty}u(x) \le f(x)$ in $X$.\\
    Fix $x\in X$, for any $\varepsilon>0$, we choose $w \in \mathcal{A}_{f,g}$ such that $0\le w(x)-u(x) \le \varepsilon$, then
    \begin{align*}
        \Delta_{\infty}u(x)
        &=\sup_{y\sim x}u(y)+\inf_{y\sim x}u(y)-2u(x)\\
        &\le \sup_{y\sim x}w(y)+\inf_{y\sim x}w(y)-2w(x)+2\varepsilon\\
        &=\Delta_{\infty}w(x) + 2\varepsilon \\
        &\le f(x)+ 2\varepsilon.
    \end{align*}
    Thus, $\Delta_{\infty}u(x) \le f(x)$.\\
    \underline{Step 3}: we prove that $\Delta_{\infty}u(x) \ge f(x)$ in $X$.\\
    Suppose that there exists a $x_0 \in X$ such that $\Delta_{\infty}u(x_0) < f(x_0)$, i.e.
    \begin{equation*}
        \sup_{y\sim x_0}u(y)+\inf_{y\sim x_0}u(y)-2u(x_0) < f(x_0).
    \end{equation*}
    Let $\Tilde{u}=u - \varepsilon \mathbb{1}_{\{x_0\}}$, where $0< \varepsilon \le \frac{1}{2}(f(x_0)-\Delta_{\infty}u(x_0))$, then $\Tilde{u} \in \mathcal{A}_{f,g}$, which is a contradiction to the definition of $u$.
\end{proof}

\subsection{Cone compare and Strong Liouville property}
Given $G=(V,E)$ and some $x_0 \in V$, the following cone functions play an important role in the discussion of infinity harmonic functions.
\begin{equation}
    \underline{C}(x):=a-bd(x_0,x),\quad \overline{C}(x):=a+bd(x_0,x),
\end{equation}
where $a \in \R$, $b\ge0$. For all $x \ne x_0$, we have 
\begin{equation}\label{equ:cone_function}
    \Delta_{\infty}\underline{C}(x) \ge 0, \quad \Delta_{\infty}\overline{C}(x) \le 0.
\end{equation}
\begin{definition}
    Let $u\in \R^{V}$, if for any subset $X\subset V$ with $\diam(X)<+\infty$, $x_0\in V\setminus X$, $a\in \R$, and $b\ge 0$,
    \begin{equation*}
        u(x)\ge \underline{C}(x)=a-b d(x_0,x), \ \forall \ x\in \partial X,
    \end{equation*}
    implies that
    \begin{equation*}
        u(x)\ge \underline{C}(x), \ \forall \ x\in \overline{X},
    \end{equation*}
    we say that $u$ satisfies the cone comparison from above, or $u \in \mathrm{CCA}(G)$. Similarly, if
    \begin{equation*}
        u(x)\le \overline{C}(x)=a+b d(x_0,x), \ \forall \ x\in \partial X,
    \end{equation*}
    implies that 
    \begin{equation*}
        u(x)\le \overline{C}(x), \ \forall \ x\in \overline{X},
    \end{equation*}
    we say that $u$ satisfies the cone comparison from below, or $u \in \mathrm{CCB}(G)$.    
\end{definition}
By Theorem~\ref{thm:CompareResult}, we have the following property.
\begin{prop}\label{prop:cone_compare}
    Let $G=(V,E)$ be a graph, $u\in \R^V$. Suppose $\Delta_{\infty}u \le 0$ (or $\Delta_{\infty}u\ge 0$ respectively), then $u\in \mathrm{CCA}(G)$ (or $\mathrm{CCB}(G)$ respectively)
\end{prop}

It is well-known that in the Euclidean space, $u\in \mathrm{CCA}$ if and only if $\Delta_{\infty}u \le 0$; $u\in \mathrm{CCB}$ if and only if $\Delta_{\infty}u \ge 0$; see \cite{crandall2001optimal}. However, in the discrete setting,  the reverse direction of Proposition~\ref{prop:cone_compare} fails. See the following example.
\begin{example}
    Consider the following function $u$ on a graph. Then $u$ satisfies both $\mathrm{CCA}$ and $\mathrm{CCB}$ for all $a\in [0,1]$. However, $-\Delta_{\infty}u \equiv 0$ if and only if $a=0.5$.
    \begin{figure}[H]
        \centering
        \begin{tikzpicture}
            \draw[gray, thick] (-1,0) -- (2,0) (-1,-1) -- (2,-1);
            \draw[gray, thick, dashed] (-2,0)--(-1,0) (-2,-1)--(-1,-1) (2,0)--(3,0) (2,-1)--(3,-1);
            \draw[gray, thick] (0.5,-0.5)--(0,0) (0.5,-0.5)--(0,-1) (0,0)--(0,-1);
            \foreach \a in {-2, 0, 2, 4}
                \filldraw (\a/2, 0) circle(2pt) node[anchor=south]{$\a$};
            \foreach \a in {-1, 1, 3, 5}
                \filldraw (\a/2-0.5, -1) circle(2pt) node[anchor=north]{$\a$};
            \filldraw (0.5,-0.5) circle(2pt) node[anchor=west]{$a$};
        \end{tikzpicture}
        \caption{}
    \end{figure}
\end{example}
As an application of Proposition~\ref{prop:cone_compare}, we prove a Liouville type theorem.
\begin{theorem}[Strong Liouville property of infinity harmonic functions]
    Let $G=(V,E)$ be connected, $u \in \R^{V}$. 
    \begin{enumerate}[(i)]
        \item If $u \ge 0$ and $\Delta_{\infty}u \le 0$, then $u\equiv const$;
        \item If $u \le 0$ and $\Delta_{\infty}u \ge 0$, then $u\equiv const$.
    \end{enumerate}
\end{theorem}
\begin{proof}
    We only prove the first conclusion, proof for the second is similar. If $\width(G) < +\infty$, the result follows directly by Theorem \ref{thm:CompareResult}. Suppose $\width(G)= +\infty$. Let $x_0, x_1 \in V$ be any two vertices with $d(x_0,x_1)=K$. For any $\varepsilon > 0$, let 
    $$\underline{C}(x) = u(x_0) - \frac{\varepsilon}{K}d(x_0,x),$$ 
    and $N > K$ large enough such that $\underline{C}(x) \le 0$ in $V \setminus B_{x_0}(N)$. Since $u \in \mathrm{CCA}(G)$, we have $u \ge \underline{C}$ in $X= B_{x_0}(N) \setminus \{x_0\}$. Then $u(x_1) \ge u(x_0) - \varepsilon$ for any $\varepsilon > 0$, which implies that $u(x_1) \ge u(x_0)$. With the same argument, we also have $u(x_0) \ge u(x_1)$, then $u\equiv const$.
\end{proof}

\section{Necessity of the conditions in Theorem~\ref{thm:ExistenceAndUniqueness}}
In this section, we study the necessity of the conditions needed in Theorem~\ref{thm:ExistenceAndUniqueness}. Here is an outline.
\begin{enumerate}[(i)]
    \item We show that ``$\width(X)<+\infty$'' is necessary for the existence of bounded solutions; see Proposition~\ref{prop:NonexistenceOfBoundedSolutions}. 
    \item We show that ``$f \ge 0$ or $f\le 0$'' is necessary for the uniqueness of bounded solutions, i.e. sign-changing $f$ may lead to multiple solutions; see Example~\ref{exa:Signchangingf}. We also refer to \cite[Section 2.2]{peres2009tug} for a probabilistic explanation.
    \item We show that the constrain ``bounded'' in the statement ``the bounded solution is unique if $f \ge 0$ or $f\le 0$'' can not be removed by constructing an unbounded solution in Example~\ref{exa:ExistenceOfUnboundedSolutions}. 
    \item We show that ``$\width(X)<+\infty$'' is necessary for the uniqueness of bounded solutions; see Example~\ref{exa:ExistenceOfMultipleBoundedSolutions} and Proposition~\ref{prop:tmp4.5}.
\end{enumerate}
\begin{prop}\label{prop:NonexistenceOfBoundedSolutions}
    Let $G=(V,E)$ be a graph, $X\subset V$ be connected with $\width(X) =+\infty$, $f\in L^{\infty}(X)$ with $\sup f < 0$, $g\in L^{\infty}(Y)$. Then the following equation admits no bounded solutions.
    \begin{equation}
        \begin{cases}
            \Delta_{\infty}u(x)=f(x), \quad &x\in X \subset V,\\
            u(x)=g(x), & x \in Y = V\setminus X.
        \end{cases}
    \end{equation}
\end{prop}
\begin{proof}
    We prove the conclusion by contradiction. Suppose that $u$ is such a bounded solution. For any $x \in X$ with $d_{G}(Y,x)=r$, $y\sim x$, by Lemma~\ref{lem:GradientEstimation}, we have
    \begin{align*}
        -2||u||_{\infty} 
        &\le u(y)-u(x)\\
        &\le \frac{u(x)-\inf{u}}{r}+\frac{(r+1)\sup{f}}{2}\\
        &\le \frac{2||u||_{\infty}}{r}+\frac{(r+1)\sup{f}}{2}.
    \end{align*}
    By letting $r \to \infty$, we get a contradiction.
\end{proof}
\begin{example}\label{exa:Signchangingf}
    Consider the following graph and function $u$. Let $\Tilde{u}:=u+a \mathbb{1}_{X}$ with $a\in [-1,1]$. Then $\Delta_{\infty}\Tilde{u}=\Delta_{\infty}u$ in $X$.
    \begin{figure}[H]
        \centering
        \begin{tikzpicture}
            \draw[thick, gray] (0,0) -- (3,0) (0,1)--(3,1) (1,0)--(1,1) (2,0)--(2,1);
            \foreach \a in {0,1,2,3}
                \filldraw (\a,0) circle(2pt) (\a,1) circle(2pt);
            \node at (0,0) [left] {$0$};
            \node at (0,1) [left] {$0$};
            \node at (3,0) [right] {$0$};
            \node at (3,1) [right] {$0$};
            \node at (1,0) [below] {$-1$};
            \node at (2,0) [below] {$-1$};
            \node at (1,1) [above] {$1$};
            \node at (2,1) [above] {$1$};
            \draw[thick, gray, dashed] (0.5,-0.5) rectangle (2.5,1.5);
            \node at (1.5,1.5) {$X$};
        \end{tikzpicture}
        \caption{}
    \end{figure}
\end{example}

\begin{example}\label{exa:ExistenceOfUnboundedSolutions}
    Consider the following graph $G=(V,E)$, where $V=\N_0$ and the edge set is defined via connecting $0\sim k \sim 2k$ for $k=1,2,3,\cdots$. Let $X=\N_0 \setminus \{0\}$. Define two functions $u(n)=n$ and $v(n)\equiv 0$, then $\Delta_{\infty}u = \Delta_{\infty}v = 0$ in $X$, with the same boundary condition $u(0)=v(0)=0$.
        \begin{figure}[H]
            \centering
            \begin{tikzpicture}
                \draw[black, dotted, ultra thick] (5,0) -- (10,0);
                \draw[gray, thick] (0,0) -- (2,0);
                \draw[gray, thick] (0,0) .. controls (1,0.5) .. (2,0);
                \draw[gray, thick] (2,0) .. controls (3,0.5) .. (4,0);
                \draw[gray, thick] (0,0) .. controls (1.5,1) .. (3,0);
                \draw[gray, thick, dashed] (3,0) .. controls (4.5,1) .. (6,0);
                \draw[gray, thick] (0,0) .. controls (2,-1) .. (4,0);
                \draw[gray, thick, dashed] (4,0) .. controls (6,-1) .. (8,0);
                \draw[gray, thick] (0,0) .. controls (2.5,-1.5) .. (5,0);
                \draw[gray, thick, dashed] (5,0) .. controls (7.5,-1.5) .. (10,0);
                \foreach \a in {0,1,2,3,4,5}
                    \filldraw (\a,0) circle(2pt) node[anchor=north]{$\a$};
            \end{tikzpicture}
            \caption{}
        \end{figure}
\end{example}
\begin{example}\label{exa:ExistenceOfMultipleBoundedSolutions}
    Authors of \cite{peres2009tug} provide an example of a tug-of-war game such that $u_{I} \ne u_{II}$, i.e. the game has no game value. We prove that multiple bounded solutions exist for that example. Given a series of positive integers $\{l_n\}_{n=0}^{\infty}$ which serve as the lengths of teeth, there is a corresponding comb graph $G=(V,E)$, where $V=\{(n,l)\in \N^0 \times \N^0: l \le l_n\}$. The edges consist forms of $(n+1,0)\sim (n,0)$ and $(n,l+1)\sim (n,l)$. 
    \begin{figure}[H]
        \centering
        \begin{tikzpicture}[scale=2]
            \draw[thick, gray] (0,0) -- (2,0) (0,0) -- (0,1) (1,0) -- (1,2) (2,0)--(2,2);
            \draw[thick, gray, dashed] (0,1)--(0,2) (1,2)--(1,3) (2,2)--(2,3);
            \draw[thick, gray, dashed] (2,0) -- (4,0);
            \filldraw (0,0) circle (1pt) node[anchor=north]{$(0,0)$};
            \filldraw (0,1) circle (1pt) node[anchor=east]{$(0,1)$};
            \filldraw (0,2) circle (1pt) node[anchor=east]{$(0,l_0)$};
            \filldraw (1,0) circle (1pt) node[anchor=north]{$(1,0)$};
            \filldraw (1,1) circle (1pt) node[anchor=west]{$(1,1)$};
            \filldraw (1,2) circle (1pt) node[anchor=west]{$(1,2)$};
            \filldraw (1,3) circle (1pt) node[anchor=west]{$(1,l_1)$};
            \filldraw (2,0) circle (1pt) node[anchor=north]{$(2,0)$};
            \filldraw (2,1) circle (1pt) node[anchor=west]{$(2,1)$};
            \filldraw (2,2) circle (1pt) node[anchor=west]{$(2,2)$};
            \filldraw (2,3) circle (1pt) node[anchor=west]{$(2,l_2)$};
        \end{tikzpicture}
        \caption{Comb graph}
    \end{figure}
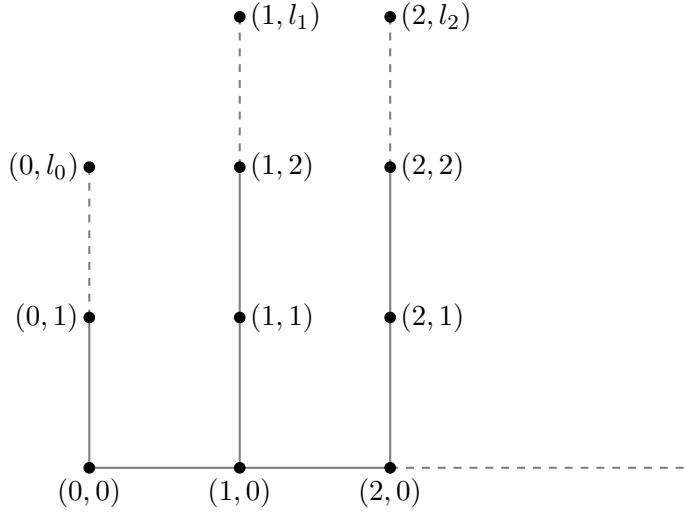
    Consider the following infinity Laplace equation
\begin{equation}\label{equ:discrete_Laplace_equation_on_a_comb}
    \begin{cases}
        \Delta_{\infty}u(n,l)=-f(n,l), \quad &l < l_n,\\
        u(n,l)=0, & l=l_n,
    \end{cases}
\end{equation}
where
\begin{equation*}
    f(n,l)=
    \begin{cases}
        \frac{1}{l_n}, \quad &l=0;\\
        0, &0< l < l_n.
    \end{cases}
\end{equation*}

Then $u(n,l)=\frac{l_n-l}{l_n}$ is a bounded solution. Define $v$ via
\begin{equation*}
    \begin{cases}
        v(0,0)=l_0\sum_{k=0}^{\infty}\frac{1}{l_k};\\
        v(n,0)=v(n-1,0)+\sum_{k=n}^{\infty}\frac{1}{l_k};\\
        v(n,l)=\frac{l_n-l}{l_n}v(n,0).
    \end{cases}
\end{equation*}
By choosing a suitable sequence $\{l_n\}_{n=0}^{\infty}$, $v$ is also a bounded solution by the following computation.
\end{example} 
\begin{prop}\label{prop:tmp4.5}
    Suppose $l_n=(n+C)^3$ with some $C$ large enough. Then the function $v$ defined above is a bounded solution of euqation~\eqref{equ:discrete_Laplace_equation_on_a_comb}.
\end{prop}
\begin{proof}
    The function $v$ is bounded by the definition. We only need to prove that
    \begin{equation}\label{equ:tmp011}
        v(n,0)-v(n-1,0) \ge \frac{1}{l_n}v(n,0)=\frac{v(n,0)}{(C+n)^3},
    \end{equation}
    i.e. function values on the shafts of the comb dominate the infinity Laplacian. By a simple calculation, we have
    \begin{equation*}
        v(0,0)=l_0\sum_{k=0}^{\infty}\frac{1}{l_k}=C^3\sum_{k=0}^{\infty}\frac{1}{(C+k)^3},
    \end{equation*}
    \begin{equation*}
        v(n,0)-v(n-1,0)=\sum_{k=n}^{\infty}\frac{1}{l_k}=\sum_{k=n}^{\infty}\frac{1}{(C+k)^3},
    \end{equation*}
    and
    \begin{align*}
        v(n,0)
        &=v(n-1,0)+\sum_{k=n}^{\infty}\frac{1}{(C+k)^3}\\
        &=v(0,0)+\sum_{i=1}^n\sum_{k=i}^{\infty}\frac{1}{(C+k)^3}\\
        &=\sum_{i=0}^{n}\frac{C^3+i}{(C+i)^3}+\sum_{k=n+1}^{\infty}\frac{C^3+n}{(C+k)^3}.
    \end{align*}
    Then \eqref{equ:tmp011} is equivalent to the following inequality
    \begin{equation*}
        \sum_{k=0}^{n}\frac{C^3+k}{(C+k)^3}+\sum_{k=n+1}^{\infty}\frac{C^3+n}{(C+k)^3}\le \sum_{k=n}^{\infty}\frac{(C+n)^3}{(C+k)^3} = \sum_{k=0}^{\infty}\frac{(C+n)^3}{(C+n+k)^3}.
    \end{equation*}
    We only need to prove that
    \begin{equation}\label{equ:tmp31}
        \frac{C^3+k}{(C+k)^3} \le \frac{(C+n)^3}{(C+n+k)^3}, \quad \forall \ 0\le k \le n,
    \end{equation}
    and 
    \begin{equation}\label{equ:tmp32}
        \frac{C^3+n}{(C+k)^3} \le \frac{(C+n)^3}{(C+n+k)^3}, \quad \forall \ k\ge n \ge 1.
    \end{equation}
    For \eqref{equ:tmp31}, we only need to prove that
    \begin{align*}
        &(n+C)^3(k+C)^3\ge (C+n+k)^3(C^3+k)\\
        \Leftarrow\ &(C^2+Ck+Cn+kn)^3\ge (C^2+Ck+Cn)^3+k(C+n+k)^3\\
        \Leftarrow\ &3kn(C^2+Ck+Cn)^2\ge k(C+n+k)^3\\
        \Leftarrow\ &3C^2n \ge C+n+k.
    \end{align*}
    The last inequality holds by choosing $C\ge 1$.\\
    For \eqref{equ:tmp32}, we only need to prove that
    \begin{align*}
        &\left(\frac{C+n+k}{C+k}\right)^3\le \frac{(C+n)^3}{C^3+n} \quad\text{(the left side is decreasing w.r.t. $k$)}\\
        \Leftarrow\ &\left(\frac{C+n+n}{C+n}\right)^3\le \frac{(C+n)^3}{C^3+n} \\
        \Leftarrow\ &(n+C)^6-(2n+C)^3(n+C^3) \ge 0\\
        \Leftarrow\ &(15C^2-8)n^4+(12C^3-12C)n^3+(3C^4-6C^2)n^2-C^3n \ge 0\\
        \Leftarrow\ &(15C^2-8)n^4+(11C^3-12C)n^3+(3C^4-6C^2)n^2+C^3(n^3-n) \ge 0.
    \end{align*}
    The last inequality holds for $C\ge 2$. Then we complete the proof.
\end{proof}

\section{Convergence of solutions of $\varepsilon$-tug-of-war games}
Let $\{\varepsilon_i\},\{G_{\varepsilon_i}\}, \{u_i\}$ be defined as previously, $W:=\width(\Omega)$. We write $d_i:=d_{G_{\varepsilon_i}}$ for short. Note that $d_i(x,y)=\lfloor\frac{d(x,y)}{\varepsilon_i}\rfloor+1$. Let $W_i:=\lfloor\frac{W}{\varepsilon_i}\rfloor+1$, which is exactly the width of $\Omega$ as a subgraph of $G_{\varepsilon_i}$.
\begin{lemma}
    The solutions $\{u_i\}$ are uniformly bounded.
\end{lemma}
\begin{proof}
    Let $\overline{q}(r)=a+br-c\frac{r(r-1)}{2}$, where $a=||g||_{\infty}$, $b=W_i\varepsilon_i^2||f||_{\infty}$, and $c=\varepsilon_i^2||f||_{\infty}$. Then $\overline{w}_i(x):=\overline{q}(d_{G_{\varepsilon_i}}(\partial\Omega, x)) \ge u_i(x)$ for all $x\in \Omega$ by Theorem~\ref{thm:CompareResult}. On the other hand, $\overline{w}_i(x) \le q(W_i)$, which is uniformly bounded. Then we get an uniform upper bound. With a similar argument, we can also get an uniform lower bound.
\end{proof}

\begin{lemma}
    For any $r>0$ there exists a $N\in \N$ such that for all $x,y\in \Omega_r$ and $x\sim y$ in $G_{\varepsilon_i}$,
    \begin{equation}
        |u_i(x)-u_i(y)| \le C_r\varepsilon_i,\quad \forall \ i>N,
    \end{equation}
    where $C_r$ is a constant depending on $r$, $||g||_{\infty}$, and $||f||_{\infty}$.
\end{lemma}
\begin{proof}
    We write $r_i:=\lfloor\frac{r}{\varepsilon_i}\rfloor$ for short. By Lemma~\ref{lem:GradientEstimation}, we have
    \begin{equation}
        |u_i(x)-u_i(y)| \le \frac{2||u_i||_{\infty}}{r_i} + \frac{r_i+1}{2}||\varepsilon_i^2f||_{\infty} \sim (\frac{2||u_i||_{\infty}}{r}+\frac{r||f||_{\infty}}{2})\varepsilon_i, \ i\to \infty.
    \end{equation}
    Since $\{u_i\}$ are uniformly bounded, then the result holds.
\end{proof}

\begin{corollary}
    For any $\delta>0$, there exists $N\in \N$ such that
    \begin{equation*}
        d_{\overline{\Omega}}(x,y) < \delta
    \end{equation*}
    implies 
    \begin{equation*}
        |u_i(x)-u_i(y)| \le C'_r\delta, \ \forall \ i>N, \ x,y\in \Omega_r,
    \end{equation*}
    where $C'_r$ is a constant depending on $r$, $||g||_{\infty}$, and $||f||_{\infty}$.
\end{corollary}
\begin{proof}
    Since $d_{\overline{\Omega}}(x,y) < \delta$, then $d_{G_{\varepsilon_i}}(x,y) < \delta_i:=\lfloor\frac{\delta}{\varepsilon_i}\rfloor+1.$ This implies that
    \begin{equation}
        |u_i(x)-u_i(y)| \le \delta_i C_r \varepsilon_i \sim \delta C_r, \ i\to \infty.
    \end{equation}
    This completes the proof.
\end{proof}

\begin{lemma}
    For any fixed $y_0\in \partial\Omega$ and $\varepsilon>0$, there exists a $\delta>0$ and $N\in \N$ such that for all $x\in \Omega$ with $d_{\overline{\Omega}}(y_0,x) \le \delta$ and $i> N$, the following estimate holds.
    \begin{equation}
        |u_i(x)-g(y_0)| \le \varepsilon. 
    \end{equation}
\end{lemma}
\begin{proof}
    Fix $\varepsilon>0$, we only need to prove that $u_i(x)\le g(y_0)+2\varepsilon$ for all $x\in \Omega$ with $d_{\overline{\Omega}}(y_0,x)< \delta$ and $i$ large enough. By the continuity of $g$, there exists $s>0$ such that $g(y) \le g(y_0)+\varepsilon$ for all $y\in \partial \Omega$ with $d_{\overline{\Omega}}(y_0,y) < s$. Let $s_i:=\lfloor\frac{s}{\varepsilon_i}\rfloor+1$, $a_i=g(y_0)+\varepsilon$, $b_i=\frac{2||g||_{\infty}}{s_i}+\varepsilon_i^2||f||_{\infty}(W_i+s_i)$ and $c_i=\varepsilon_i^2||f||_{\infty}$. Define $\overline{w}_i(x):=\overline{q}(d'_i(y_0,x)) \ge u_i(x)$ as in \eqref{equ:quadratic_function_on_new_graph} with $R=s_i$. Since $d'_i(y_0, x) \leq d_i(y_0, x)$ and $b_id_i(y_0, x) \to \left(\frac{2\|g\|_{\infty}}{s} + (W+s)||f||_{\infty}\right)d_{\overline{\Omega}}(y_0, x)$ as $i\to \infty$, then we have
    \begin{equation}
        \limsup_{i\to \infty} u_i(x) \le g(y_0)+\varepsilon+\left(\frac{2\|g\|_{\infty}}{s} + (W+s)||f||_{\infty}\right)d_{\overline{\Omega}}(y_0,x).
    \end{equation}
    If we choose $\delta$ small enough, then for all $x\in \Omega$ with $d_{\overline{\Omega}}(y_0,x)< \delta$ and $i$ large enough, we have $u_i(x)\le g(y_0)+2\varepsilon$. This completes the proof.
\end{proof}

Now we are ready to proof Theorem~\ref{thm:Convergence}.
\begin{proof}[Proof of Theorem~\ref{thm:Convergence}]We prove the result by the proof of Arzel\`a–Ascoli Theorem.

    \underline{Step 1}: We construct $u$. Let $\{x_k\}_{k=1}^{\infty}$ be a countable dense subset of $\Omega$, by a diagonal method, there exists a subsequence of $\{u_i\}$, denoted still by $\{u_i\}$ for convenience, converges on every $x_k$. Next we prove that $\{u_i(x)\}$ converges on every $x\in \Omega$. Let $d_{\overline{\Omega}}(\partial\Omega,x)=r$, there exists a $x_k$ such that $d_{\overline{\Omega}}(\partial\Omega,x_k) < r$ and $d(x_k,x)=\delta$ with $\delta$ small enough. Then there exists a $N\in\N$ such that for all $i,j > N$,
    \begin{align*}
        |u_i(x)-u_j(x)|
        &\le |u_i(x)-u_i(x_k)|+|u_i(x_k)-u_j(x_k)|+|u_j(x_k)-u_j(x)|\\
        &\le C_r \delta + \varepsilon + C_r \delta.
    \end{align*}
    Choose $\delta$ small enough, then we get the convergence. Let 
    \begin{align*}
        u(x)=
        \begin{cases}
            \lim\limits_{i\to \infty}u_i(x), \quad &x\in \Omega,\\
            g(x), &x\in \partial \Omega.
        \end{cases}
    \end{align*}
    \\
    \underline{Step 2}: We prove that the convergence is locally uniformly. For any compact $K\subset \Omega$ with $d_{\overline{\Omega}}(\partial \Omega, K)=r>0$, let $\{x_k\}_{k=1}^M$ be a subset of $K$ such that $\bigcup_{k=1}^{M}B_{x_k}(\delta) \supset K$, where $\delta$ is small enough. Then we have 
    \begin{align*}
        |u_i(x)-u_j(x)|
        &\le |u_i(x)-u_i(x_k)|+|u_i(x_k)-u_j(x_k)|+|u_j(x_k)-u_j(x)|\\
        &\le C_r \delta + \varepsilon + C_r \delta.
    \end{align*}
    Choose $\delta$ small enough, then we get the result.\\
    \underline{Step 3}: We prove that $u\in C(\overline{\Omega})$. For any $r>0$ and $x,y\in \Omega_r$, there is a $N\in \N$ large engouth such that
    \begin{align*}
        |u(x)-u(y)| 
        &\le |u(x)-u_N(x)|+|u_N(x)-u_N(y)|+|u_N(y)-u(y)|\\
        &\le \varepsilon + C_rd_{\overline{\Omega}}(x,y) + \varepsilon.
    \end{align*}
    This indicates $u\in C(\Omega)$. Together with Lemma 5.4, we obtain $u\in C(\overline{\Omega})$.
\end{proof}

\section*{Acknowledgements}
The authors are grateful to Prof. Bobo Hua for his guidance and support. The authors would also like to thank Florentin M\"unch and Genggeng Huang for their helpful advice.

\bibliographystyle{alpha}
\bibliography{bib}

\end{document}